\documentclass[11pt,a4paper]{amsart}
\usepackage{amscd,amssymb,amsopn,amsmath,amsthm,graphics,amsfonts,enumerate,verbatim,calc}
\usepackage[dvips]{graphicx}

\usepackage{color}

\usepackage{amssymb,amsmath}

\headsep=1cm \numberwithin{equation}{section}
\newtheorem{theorem}{Theorem}[section]
\newtheorem{lemma}[theorem]{Lemma}
\newtheorem{proposition}[theorem]{Proposition}
\newtheorem{corollary}[theorem]{Corollary}

\theoremstyle{definition}
\newtheorem{definition}[theorem]{Definition}

\theoremstyle{remark}
\newtheorem{remark}[theorem]{Remark}
\newtheorem{example}[theorem]{Example}

\newtheorem{acknowledgement}{Acknowledgement}

\newcommand{\Spec}{\operatorname{Spec}}

\newcommand{\Ht}{\operatorname{ht}}
\newcommand{\CH}{\operatorname{CH}}

\newcommand{\depth}{\operatorname{depth}}

\newcommand{\Cl}{\operatorname{Cl}}

\newcommand{\Frac}{\operatorname{Frac}}

\newcommand{\Sp}{\operatorname{Sp}}
\newcommand{\Proj}{\operatorname{Proj}}
\newcommand{\Reg}{\operatorname{Reg}}
\newcommand{\Nor}{\operatorname{Nor}}

\newcommand{\fm}{\frak{m}}
\newcommand{\fp}{\frak{p}}
\newcommand{\fq}{\frak{q}}

\begin{document}
\title[Normal hyperplane sections of normal schemes]
{Normal hyperplane sections of normal schemes in mixed characteristic}

\author[J. Horiuchi]{Jun Horiuchi}
\address{Department of Mathematics, Nippon Institute of Technology, Miyashiro, Saitama 345-8501, Japan}
\email{jhoriuchi.math@gmail.com}

\author[K. Shimomoto]{Kazuma Shimomoto}
\address{Department of Mathematics, College of Humanities and Sciences, Nihon University, Setagaya-ku, Tokyo 156-8550, Japan}
\email{shimomotokazuma@gmail.com}

\thanks{2010 {\em Mathematics Subject Classification\/}: 12F10, 13B40, 13J10, 14G40, 14H20}

\keywords{Bertini theorem, graded ring, hyperplane section, normal scheme}


\begin{abstract}
The aim of this article is to prove that, under certain conditions, an affine flat normal scheme that is of finite type over a local Dedekind scheme in mixed characteristic admits infinitely many normal effective Cartier divisors. For the proof of this result, we prove the Bertini theorem for normal schemes of some type. We apply the main result to prove a result on the restriction map of divisor class groups of Grothendieck-Lefschetz type in mixed characteristic.
\end{abstract}

\maketitle

\begin{center}
{\textit{Dedicated to Prof. Gennady Lyubeznik on the occasion of his 60th birthday.}}
\end{center}

\section{Introduction}

Let $X$ be a connected Noetherian normal scheme. Then does $X$ have sufficiently many normal Cartier divisors? The existence of such a divisor when $X$ is a normal projective variety over an algebraically closed field is already known. This fact follows from the classical Bertini theorem due to Seidenberg. In birational geometry, it is often necessary to compare the singularities of $X$ with the singularities of a divisor $D \subset X$, which is known as \textit{adjunction} (see \cite{KM98} for this topic). In this article, we prove some results related to this problem in the mixed characteristic case. The main result is formulated as follows (see Corollary \ref{normalCartier} and Remark \ref{normalCartier2}).

\begin{theorem}
Let $X$ be a normal connected affine scheme such that there is a surjective flat morphism of finite type $X \to \Spec A$, where $A$ is an unramified discrete valuation ring of mixed characteristic $p>0$. Assume that $\dim X \ge 2$, the generic fiber of $X \to \Spec A$ is geometrically connected and the residue field of $A$ is infinite. 

Then there exists an infinite family of pairwise distinct effective Cartier divisors $\{D_{\lambda}\}_{\lambda \in \Lambda}$ of $X$ such that each $D_{\lambda}$ is normal and flat over $A$.
\end{theorem}

The proof of the theorem is obtained as an immediate corollary of the\textit{ Bertini type theorem} for normal schemes in mixed characteristic. The hypothesis that $A$ is \textit{unramified} is forced in the proof of the local Bertini theorem in \cite{OS15}, where we needed to estimate the bound of the number of minimal generators of the (completed) module of K$\mathrm{\ddot{a}}$hler differentials over an unramified complete discrete valuation ring (see Proposition \ref{symbolic} below). However, we believe that the theorem holds true for an arbitrary discrete valuation ring with infinite residue field
(see also Remark \ref{normalCartier2}). Let $X \subset \mathbb{P}^n_K$ be a projective variety over an algebraically closed field $K$. Let $\mathbf{P}$ be one of the following properties: smooth, normal or reduced. Assume that $X$ has $\mathbf{P}$. Then is it true that the scheme theoretic intersection $X \cap H$ has the same property for a generic hyperplane $H \subset \mathbb{P}^n_K$? It is a classical theorem of Bertini that the generic hyperplane section of a smooth projective variety $X \subset \mathbb{P}^n_K$ is smooth. It is also a classical result of Seidenberg \cite{S50} that, if $X \subset \mathbb{P}^n_K$ is a normal connected variety, then $X \cap H$ is a normal connected variety for a generic hyperplane $H \subset \mathbb{P}^n_K$. Now let $X$ be a scheme such that there is a flat surjective morphism of finite type $X \to S$ and $S$ is a Dedekind scheme of mixed characteristic.\footnote{$X \to S$ is usually called an \textit{arithmetic scheme}. However, we will not use this terminology in the article, because it seems that different authors use it in slightly different settings.} The main reason that we assume $X \to S$ to be surjective is that we want to avoid to include the following case: $X=\Spec \mathbb{Q}_p \to S=\Spec \mathbb{Z}_p$ and $\mathbb{Q}_p \simeq \mathbb{Z}_p[X]/(pX-1)$. We prove the \textit{global Bertini theorem} when $X$ is a normal connected scheme (see Theorem \ref{mainBertini} and Corollary \ref{corBertini} when $X$ is a regular connected scheme).

\begin{theorem}
Let $X$ be a normal connected projective scheme over $\Spec A$ such that $X \to \Spec A$ is flat and surjective, where $(A,\pi_A,k)$ is an unramified discrete valuation ring of mixed characteristic $p>0$. Assume that $\dim X \ge 2$, the generic fiber of $X \to \Spec A$ is geometrically connected and $k$ is an infinite field. 

Then there exists an embedding $X \hookrightarrow \mathbb{P}^n_A$, together with a Zariski-dense open subset $U \subset \mathbb{P}^n(k)^{\vee}$, where $\mathbb{P}^n(k)^{\vee}$ is the dual projective space of $\mathbb{P}^n(k)$, such that the scheme theoretic intersection $X \cap H$ is normal and flat over $A$ for any $H \in \Sp_A^{-1}(U)$, where $\Sp_A:\mathbb{P}^n(A)^{\vee} \to \mathbb{P}^n(k)^{\vee}$ is the specialization map (see Definition \ref{spmap}).
\end{theorem}

A crucial ingredient for the proof the above theorem is Lemma \ref{normalgraded}; a construction of certain Noetherian graded normal domains. As an application, we prove a mixed characteristic version of a result of Ravindra and Srinivas on Weil divisor class groups on normal schemes in the final section (see Corollary \ref{RavindraSrinivas}). We intend that the present paper provides the users with basic tools to study the connection between the Bertini-type problem for local (graded) rings and the Bertini-type problem for projective schemes.

\section{Notation and conventions}

All rings are commutative and unitary.  We refer the reader to \cite{BrHer93} and \cite{M86} as standard references for commutative algebra. A \textit{local ring} is a Noetherian ring with a unique maximal ideal. Let $R$ be a Noetherian ring. Let us denote by $\Reg(R)$ the regular locus of $\Spec R$ and denote by $\Nor(R)$ the normal locus of $\Spec R$. The symbol $(R,\fm,k)$ will denote a local ring such that $k=R/\fm$ is the residue field. Let $(A,\pi_A,k)$ be a discrete valuation ring of mixed characteristic $p>0$ such that $\pi_A$ is a uniformizer of $A$ and $p \notin \pi_A^2A$. Such a discrete valuation ring $A$ is called \textit{unramified}. A ring map $(A,\pi_A,k) \to (R,\fm,k)$ of complete local rings is called a \textit{coefficient ring map}, if it is a local flat map, $A$ is an unramified complete discrete valuation ring of mixed characteristic and $k=A/\pi_A A \simeq R/\fm$. Assume that both $(R,\fm,k)$ and $(A,\pi_A,k)$ are complete local rings. Let $\Omega_{R/A}$ denote the module of K$\mathrm{\ddot{a}}$hler differentials of $R$ over $A$. It is known that $\Omega_{R/A}$ is not a finitely generated $R$-module when $\dim R \ge 2$, for which we refer the reader to \cite[Examples 5.5 (a)]{K86}.\footnote{The main point is that when $R$ is a domain, then the transcendence degree of the field extension $\Frac(A) \to \Frac(R)$ is infinite and apply \cite[Corollary 5.3]{K86}.} Let $\widehat{\Omega}_{R/A}$ be the $\fm$-adic completion of $\Omega_{R/A}$. Then $\widehat{\Omega}_{R/A}$ is a finitely generated $R$-module (see \cite[Proposition 20.7.15]{Gr64} for the proof of this fact). The book \cite{K86} is a good source for differential modules. In the present article, the completed module $\widehat{\Omega}_{R/A}$ will play a vital role. Let $R$ be a ring and $I$ be its ideal. Then let $V(I)$ denote the set of all primes of $R$ containing the ideal $I$. Let us put $D(I):=\Spec R \setminus V(I)$. Let $(R,\fm,k)$ be a local ring. Then we say that the elements $x_1,\ldots,x_n$ in $R$ are \textit{minimal generators} of $\fm$, if they span the $k=R/\fm$-vector space $\fm/\fm^2$.
 
We follow \cite{GW10} for the notation in algebraic geometry. When we speak of a divisor, it always means a \textit{Cartier divisor}. A Weil divisor is used only when we discuss the \textit{Weil divisor class group} on normal schemes (see \cite[Chapter 11]{GW10} for divisors and Weil divisor class groups).

\section{Localization of an affine cone attached to a projective scheme}

In this section, we recall basic theory on Noetherian graded rings and introduce some geometric method using the \textit{affine cone} of projective schemes. The notion of affine cones will be necessary to relate the global Bertini theorem to the local Bertini theorem in mixed characteristic as proved in \cite{OS15}. A general reference for graded rings is \cite{BrHer93}. Let $R=\bigoplus_{n \in \mathbb{Z}} R_n$ be a $\mathbb{Z}$-graded Noetherian ring with a prime ideal $\fp$. Let $\fp^*$ be the homogeneous ideal of $R$ which is generated by homogeneous elements contained in $\fp$. It is known that $\fp^*$ is a prime ideal (see \cite[Lemma 1.5.6]{BrHer93} for the proof of this fact). We denote by $R_{(\fp)}$ the \textit{homogeneous localization} of $R$ with respect to $\fp$. That is, the $n$-th graded piece of $R_{(\fp)}$ is given  by
$$
(R_{(\fp)})_n=\Big\{\frac{a}{b} \in R_{(\fp)}~\Big|~a~\mbox{and}~b~\mbox{are homogeneous,}~\deg a-\deg b=n\Big\}.
$$
Then $R_{(\fp)}$ is a subring of $R_{\fp}$. We consider the following condition on graded rings.

\begin{enumerate}
\item[$(\mathbf{St})$]
Let $R=\bigoplus_{n \ge 0} R_n$ be a positively graded Noetherian ring such that $R$ is generated by $R_1$ over $R_0$, where ($R_0,\fm_0,k)$ is a local ring with the maximal ideal $\fm_0$.
\end{enumerate}

A graded ring satisfying the condition $(\mathbf{St})$ is a special case of \textit{standard graded rings}. In practice, such a graded ring arises as the homogeneous coordinate ring of a projective scheme. Let $X \subset \mathbb{P}^n_{R_0}=\Proj \big(R_0[X_0,\ldots,X_n]\big)$ be a closed subscheme over $\Spec R_0$. Then there is a homogeneous ideal $I$ such that $X \simeq \Proj R$, where $R=R_0[X_0,\ldots,X_n]/I$ (see \cite[Proposition 13.24]{GW10}). Then $R$ satisfies $(\mathbf{St})$. We put $\fm:=\fm_0 \oplus \bigoplus_{n>0}R_n$. Since $R_0$ is a local ring, $\fm$ is the unique homogeneous maximal ideal of $R$. We also let $\fm_+:=\bigoplus_{n>0} R_n$, which is the \textit{irrelevant ideal} of $R$. We define the \textit{affine cone} by $C(R):=\Spec R$. Its \textit{pointed affine cone} is defined by $C^0(R):=C(R) \setminus i(V)$, where $i:V=\Spec R_0 \hookrightarrow C(R)$ is the closed immersion induced by $R \twoheadrightarrow R/\fm_+=R_0$. Let $R_{\fm}$ be the localization of $R$ at the maximal ideal $\fm$. Then we define the \textit{localized pointed affine cone} by $C^0_+(R):=\Spec R_{\fm} \setminus j(V)$, where $j:V \to \Spec R_{\fm}$ is induced by $R_{\fm} \twoheadrightarrow R_{\fm}/\fm_+ R_{\fm}=R_0$. For a homogeneous ideal $I \subset R$, we set $D_+(I)$ to be the set of all primes $\fp \in \Proj R$ such that $\fp$ does not contain $I$. The following proposition is found in  \cite[Proposition 13.37]{GW10}. However, we give its proof for the convenience of the readers.

\begin{proposition}
\label{Cone}
Let the notation be as above. Then there is a sequence of morphisms of schemes:
$$
C_+^0(R) \xrightarrow{\psi} C^0(R) \xrightarrow{\phi} X.
$$
The following properties hold:
\begin{enumerate}
\item
$\phi \circ \psi$ is a flat surjective morphism with smooth fibers of relative dimension one.

\item
Assume that $\mathbf{P}$ is Serre's condition $(R_n)$ or $(S_n)$. Then $C_+^0(R)$ has $\mathbf{P}$ if and only if so does $X$.
\end{enumerate}
\end{proposition}

\begin{proof}
Since $R$ is generated by $R_1$ over $R_0$ by our assumption, there exists a finite set of elements $f_1,\ldots,f_s \in R_1$ such that $X=\bigcup_{i=1}^n D_+(f_i)$ and since each $R_{f_i}$ has an invertible element $f_i$ of degree 1, there is an isomorphism:
$$
(R_{f_i})_0[T,T^{-1}] \cong R_{f_i}~(T \mapsto f_i).
$$
Using this description, since $\phi$ is locally induced by the natural inclusion $(R_{f_i})_0 \hookrightarrow R_{f_i}$ and $\psi$ is induced by the localization map $R \to R_{\fm}$, it follows that $\phi \circ \psi$ is flat.  Moreover, for any prime ideal $\fp \subset (R_{f_i})_0$, the extension $\fp R_{f_i}$ is also prime. Hence $\phi$ is surjective. Every point $ \fp \in X$ is contained in $\fm$, but not containing $\fm_+$ and so we find that $\phi \circ \psi$ is surjective. On the other hand, we have
$$
\phi^{-1}(D_+(f_i))=D_+(f_i) \times_{\Spec R} \Spec \big(R[T,T^{-1}]\big),
$$
where $R[T]$ is a polynomial algebra over $R$ with a variable $T$. That is, the fiber of $\phi$ at the point $\fp \in X$ is the punctured affine line $\mathbb{A}^1_{k(\fp)} \setminus \{0\}=\Spec \big(k(\fp)[T]\big) \setminus \{0\}$, where $k(\fp)=(R_{(\fp)})_0/\fp_0 (R_{(\fp)})_0$. Moreover, we have $(\phi \circ \psi)^{-1}(\fp) \simeq \Spec \big(k(\fp)[T]_{(T)}\big) \setminus \{0\}$, which is a smooth scheme over $\Spec k(\fp)$. We have thus proved the assertion $\rm(1)$. For the assertion $\rm(2)$, it suffices to apply \cite[Theorem 23.9]{M86}.
\end{proof}

\begin{example}
Let $k$ be a field and consider $X=\mathbb{P}^n_k=\Proj \big(k[X_0,\ldots,X_n]\big)$. Then we have $C^0(R)=\mathbb{A}^{n+1}_k \setminus \{(0,\ldots,0)\}$, where $(0,\ldots,0)$ is the origin of $\mathbb{A}^{n+1}_k$. Then $\mathbb{A}^{n+1}_k \setminus \{(0,\ldots,0)\} \to \mathbb{P}^n_k$ gives a construction of the projective space as a set of lines through the origin of $\mathbb{A}^{n+1}_k$. Moreover, we have $C_+^0(R)=\Spec \big(k[X_0,\ldots,X_n]_{(X_0,\ldots,X_n)}\big) \setminus \{0,\ldots,0\}$.
\end{example}

\section{Basic elements and symbolic power ideals}

For the proof of the main theorem of the article, let us prepare notation from \cite{OS15}. Especially, the notion of \textit{basic elements} plays an important role.

\begin{definition}
Let $R$ be a Noetherian ring and let $M$ be a (not necessarily finitely generated) $R$-module. Then we say that an element $m \in M$ is \textit{basic} at a prime ideal $\fp$ of $R$, if the image of $m$ under the map $M \to M_{\fp}/\fp M_{\fp}$ is not zero (in particular, $M_{\fp} \ne 0$). 
\end{definition}

We will also need the derivation for modules.

\begin{definition}
Let $R$ be a ring and let $M$ be an $R$-module. A set-theoretic map $d:R \to M$ is called a \textit{derivation}, if the following equalities are satisfied: $d(a+b)=da+db$ and $d(ab)=a(db)+b(da)$ for any elements $a,b \in R$. 
\end{definition}

The importance of basic elements is expressed by the following lemma (see \cite[Lemma 2.2]{Fl77} for the proof).

\begin{lemma}
\label{basiclemma}
Let $d:R \to M$ be a derivation and let $a \in R$. Assume that $da \in M$ is basic at a prime $\fp$ of $R$. Then we have $a \notin \fp^{(2)}$, where $\fp^{(n)}:=\fp^n R_{\fp} \cap R$ is the $n$-th symbolic power ideal of $\fp$.
\end{lemma}

We will need the lemma in the case that $M$ is the (completed) module of K$\mathrm{\ddot{a}}$hler differentials and $d$ is the canonical derivation. Let $(A,\pi_A,k)$ be a discrete valuation ring of mixed characteristic $p>0$. Let $\mathbb{P}^n(A)$ denote the $n$-dimensional projective space with coordinates in $A$ such that its $A$-rational point $(\alpha_0:\cdots:\alpha_n) \in \mathbb{P}^n(A)$ is \textit{normalized}, which means that $\pi_A \nmid \alpha_i$ for some $0 \le i \le n$. By this convention, we have $(\alpha_0:\cdots:\alpha_n)=(\beta_0:\cdots:\beta_n)$ in $\mathbb{P}^n(A)$ if and only if the equality $\beta_i=v \alpha_i$ holds for some $v \in A^{\times}$ and all $0 \le i \le n$. We notice that $\mathbb{P}^n(A)$ may be identified naturally with $\mathbb{P}^n(K)$ as sets, where $K$ is the field of fractions of $A$.

\begin{definition}[Specialization map]
\label{spmap}
Let us define the set-theoretic map $\Sp_A:\mathbb{P}^n(A) \to \mathbb{P}^n(k)$ in the following way. Let us pick a point $\alpha=(\alpha_0:\dots:\alpha_n) \in \mathbb{P}^n(A)$  with its lift $\widetilde{\alpha}=(\widetilde{\alpha}_0,\ldots,\widetilde{\alpha}_n) \in A^{n+1} \setminus \{0,\ldots,0\}$. Then we define
$$
\Sp_A(\alpha):=(\overline{\alpha}_0:\cdots:\overline{\alpha}_n) \in \mathbb{P}^n(k),
$$
where we put $\overline{\alpha}_i:=\widetilde{\alpha}_i \pmod {\pi_A}$. 
\end{definition}

Every point of $\mathbb{P}^n(A)$ is normalized and it is easy to check that this map is independent of the lift of $\alpha=(\alpha_0:\cdots:\alpha_n)$. Thus, the specialization map is well defined. Let $(R,\fm,k)$ be a Noetherian local $A$-algebra and fix a system of elements $x_0,\ldots,x_n$ in the maximal ideal $\fm$ and let us choose a point $\alpha=(\alpha_0:\cdots:\alpha_n) \in \mathbb{P}^n(A)$. Let us put
$$
\mathbf{x}_{\widetilde{\alpha}}:=\sum_{i=0}^n \widetilde{\alpha}_i x_i,
$$
where $\widetilde{\alpha}=(\widetilde{\alpha}_0,\ldots,\widetilde{\alpha}_n) \in A^{n+1} \setminus \{0,\ldots,0\}$ is a lift of $\alpha=(\alpha_0:\cdots:\alpha_n) \in \mathbb{P}^n(A)$ through the quotient map $A^{n+1} \setminus \{0,\ldots,0\} \to \mathbb{P}^n(A)$. The principal ideal $\mathbf{x}_{\widetilde{\alpha}}R$ does not depend on the lift of $\alpha \in \mathbb{P}^n(A)$.

The following lemma is a modification of \cite[Lemma 4.2]{OS15}, which plays a role in the proof of the Bertini theorem. The proof given in \cite[Lemma 4.2]{OS15} applies without any essential change.

\begin{lemma}
\label{Avoidance}
Let $(R,\fm,k)$ be a complete Noetherian local domain of mixed characteristic $p > 0$ with infinite residue field $k$ and a coefficient ring $(A,\pi_A,k)$. Fix a set of elements $x_0,\ldots,x_n$ in the maximal ideal $\fm$, together with a prime ideal $\fp$ of $R$ such that $(x_0,\ldots,x_n) \not\subset \fp$. Then there exists a Zariski-dense open subset $U \subset \mathbb{P}^n(k)$ such that $\mathbf{x}_{\widetilde{\alpha}} \notin \fp$ for every point $\alpha=(\alpha_0:\cdots:\alpha_n) \in \Sp_A^{-1}(U)$.
\end{lemma}

We shall prove the following proposition which slightly generalizes \cite[Theorem 4.3]{OS15} in the form we need. We explain the meaning of the conclusion after the proof is finished.

\begin{proposition}
\label{symbolic}
Let $(R,\fm,k)$ be a complete Noetherian local domain of mixed characteristic $p > 0$ and assume the following conditions:

\begin{enumerate}
\item[\rm{(i)}]
Let $A \to R$ be a coefficient ring map, where $(A,\pi_A,k)$ is an unramified complete discrete valuation ring.

\item[\rm{(ii)}]
Let $\pi_A,x_1,\ldots,x_d$ be a fixed set of minimal generators of $\fm$, which in particular implies that $\pi_A \notin \fm^2$.

\item[\rm{(iii)}]
The residue field $k$ is infinite.
\end{enumerate}
Then there exist only finitely many prime ideals $\{\fq_1,\ldots,\fq_n\}$ of $R$ such that the ideal $(x_1,\ldots,x_d)$ is contained in every $\fq_i$ with $i=1,\ldots,n$. Let us put
$$
\mathbf{x}_{\widetilde{\alpha}}=\sum_{i=1}^d \widetilde{\alpha}_i x_i~\mbox{for}~\alpha=(\alpha_1:\cdots:\alpha_d) \in \mathbb{P}^{d-1}(A).
$$
Then, there is a Zariski-dense open subset $U \subset \mathbb{P}^{d-1}(k)$ such that $\mathbf{x}_{\widetilde{\alpha}} \notin \fp^{(2)}$ holds for every $\fp \in \Spec R \setminus \{\fq_1,\ldots,\fq_n\}$ and $\alpha \in \Sp^{-1}_A(U)$.
\end{proposition}

\begin{proof}
For the first assertion, we observe that $\Ht(x_1,\ldots,x_d) \ge \dim R-1$ in view of the equality $\Ht(\pi_A,x_1,\ldots,x_d)=\dim R$. This implies that the ideal $(x_1,\ldots,x_d)$ is contained in only finitely many primes $\fq_1,\ldots,\fq_n$ and $\fm \in \{\fq_1,\ldots,\fq_n\}$.

Let us prove the second assertion. Let $\widehat{\Omega}_{R/A}$ be the $\fm$-adic completion of the module of K$\mathrm{\ddot{a}}$hler differentials of the $A$-algebra $R$. Since $\pi_A \in A$, we have $d\pi_A=0$ in $\widehat{\Omega}_{R/A}$, where $d:R \to \widehat{\Omega}_{R/A}$ is the universal derivation. So it follows that $dx_1,\ldots,dx_d$ forms a set of generators of the $R$-module $\widehat{\Omega}_{R/A}$. To prove the second assertion, it is necessary to modify the proof of \cite[Theorem 4.3]{OS15} as follows. Let us put $x_0:=\pi_A$ and let $R[X_0,\ldots,X_d]$ be a polynomial algebra over $R$. Then by the proof of \cite[Theorem 4.3]{OS15}, there exist finitely many prime ideals $\{\fp_1,\ldots,\fp_s\}$ of $R$, together with a polynomial 
$$
G \in R[X_0,\ldots,X_d] \setminus \fm R[X_0,\ldots,X_d]
$$
such that for any point $\widetilde{\alpha}=(\widetilde{\alpha}_0,\ldots,\widetilde{\alpha}_d) \in A^{d+1}$ with $G(\widetilde{\alpha}) \notin \fm$, we have
$$
\sum_{i=0}^d \widetilde{\alpha}_i dx_i \in \widehat{\Omega}_{R/A}~\mbox{is basic at every}~\fp \in \Spec R \setminus \{\fp_1,\ldots,\fp_s\}.
$$
Let $\{\fq_1,\ldots,\fq_n\}$ be the set of primes of $R$ as in the first assertion. After organizing the set of the primes $\{\fp_1,\ldots,\fp_s\}$ in an appropriate order, we may assume that there exists an integer $r$ such that $0 \le r \le s$ and $\{\fp_1,\ldots,\fp_r\}=\{\fp_1,\ldots,\fp_s\} \setminus \{\fq_1,\ldots,\fq_n\}$. Then we have $(x_1,\ldots,x_d) \not\subset \fp$ when $\fp \in \{\fp_1,\ldots,\fp_r\}$, and $(x_1,\ldots,x_d) \subset \fp$ when $\fp \in \{\fp_{r+1},\ldots,\fp_s\}$. 

By proceeding as in the second part of the proof of \cite[Theorem 4.3]{OS15} and applying Lemma \ref{Avoidance} to the prime ideals $\{\fp_1,\ldots,\fp_r\}$, we can find a Zariski-dense open subset $U' \subset \mathbb{P}^{d}(k)$ ($U'$ is determined by $G \in R[X_0,\ldots,X_d]$) such that for a given point $\alpha=(\alpha_0:\cdots:\alpha_d) \in \Sp^{-1}_A(U')$, the following assertion holds:
\begin{equation}
\label{crucial}
\sum_{i=1}^d \widetilde{\alpha}_i dx_i \in \widehat{\Omega}_{R/A}~\mbox{is basic at every}~\fp \in \Spec R \setminus \{\fp_1,\ldots,\fp_s\}
\end{equation}
$$
\mbox{and}~\mathbf{x}_{\widetilde{\alpha}}=\sum_{i=1}^d \widetilde{\alpha}_i x_i \notin \fp~\mbox{for every}~\fp \in \{\fp_1,\ldots,\fp_r\}
$$
(notice that $dx_0=d\pi_A=0$). Then applying Lemma \ref{basiclemma} to $(\ref{crucial})$, we deduce the following assertion:
\begin{equation}
\label{crucial2}
\mathbf{x}_{\widetilde{\alpha}}=\sum_{i=1}^{d} \widetilde{\alpha}_i x_i \notin \fp^{(2)}~\mbox{for every}~\fp \in \Spec R \setminus \{\fp_{r+1},\ldots,\fp_s\}
\end{equation}
$$
\mbox{and every}~\alpha=(\alpha_0:\cdots:\alpha_d) \in \Sp^{-1}_A(U').
$$
Let us identify the hyperplane $H \subset \mathbb{P}^d(k)$, which is defined by the homogeneous equation $X_0=0$ in $k[X_0,\ldots,X_d]$, with the projective space $\mathbb{P}^{d-1}(k)$. Then we have an open subset $U:=U' \cap H$ of $H \cong \mathbb{P}^{d-1}(k)$, which is not empty by the construction. Since we have $\{\fp_{r+1},\ldots,\fp_s\} \subset \{\fq_1,\ldots,\fq_n\}$, it follows that $\Spec R \setminus \{\fq_1,\ldots,\fq_n\} \subset \Spec R \setminus \{\fp_{r+1},\ldots,\fp_s\}$ and the second assertion follows from $(\ref{crucial2})$.
\end{proof}

\begin{remark}
\label{symbolic2}
Let us explain the meaning of non-containment in the second symbolic power ideal $\fp^{(2)}$. Let $R$ be a Noetherian ring and fix a prime ideal $\fp$ of $R$ such that
$R_{\fp}$ is regular. Choose an element $x \in \fp$. Then we can prove that the localization of $R/xR$ at $\fp$ is regular if $x \notin \fp^{(2)}$. Indeed, the condition $x \notin \fp^{(2)}$ implies that $x$ is a minimal generator of the maximal ideal of $R_{\fp}$. Since $R_{\fp}$ is regular, $R_{\fp}/xR_{\fp}$ is also regular. Thus, Proposition \ref{symbolic} asserts that the inclusion $\Reg(R) \cap V(\mathbf{x}_{\widetilde{\alpha}} R) \subset \Reg(R/\mathbf{x}_{\widetilde{\alpha}}R)$ holds outside finitely many points $\fq_1,\ldots,\fq_n$ of $\Spec R$.
\end{remark}

\section{Proof of the main theorem}

We need the following lemma. We refer the reader to \cite[Corollaire (2.9.4)]{Gr61} for the proof and its generalized version.

\begin{lemma}
\label{Gro}
Assume that $A$ is a Noetherian ring and let $X \hookrightarrow \mathbb{P}_A^N$ be a closed subscheme with its ideal sheaf $\mathcal{I}_X$, and let $\Gamma_*(\mathcal{I}_X):=\bigoplus_{n \ge 0}H^0(\mathbb{P}_A^N,\mathcal{I}_X(n))$ be the homogeneous ideal of the polynomial algebra $\Gamma_*(\mathcal{O}_{\mathbb{P}^N_A})$. Then $X$ is an integral scheme if and only if $\Gamma_*(\mathcal{I}_X)$ is a prime ideal of $\Gamma_*(\mathcal{O}_{\mathbb{P}^N_A})$.
\end{lemma}

\begin{lemma}
\label{normalgraded}
Let $X$ be a normal connected projective scheme over $\Spec A$ such that $X \to \Spec A$ is flat and surjective, where $A$ is a Noetherian local normal domain. Suppose that one of the following conditions holds:
\begin{enumerate}
\item
The field of fractions of $A$ is of characteristic 0 and the generic fiber of $X \to \Spec A$ is geometrically connected.

\item
The generic fiber of $X \to \Spec A$ is geometrically integral.
\end{enumerate}
Then there exists a Noetherian graded normal domain $R=\bigoplus_{n \ge 0} R_n$ such that $R$ satisfies the condition $(\mathbf{St})$, $R_0=A$ and $X \simeq \Proj R$.
\end{lemma}

\begin{proof}
Since $X$ is a Noetherian normal connected scheme, it is an integral scheme. Fix an embedding $X \hookrightarrow \mathbb{P}_A^N$ as a closed subscheme with its ideal sheaf $\mathcal{I}_X$. By Lemma \ref{Gro}, $S:=\Gamma_*(\mathcal{O}_{\mathbb{P}^N_A})/\Gamma_*(\mathcal{I}_X)$ is a Noetherian graded domain satisfying $(\mathbf{St})$ such that $X \simeq \Proj S$. Let us put $R':=\Gamma_*(\mathcal{O}_X)=\bigoplus_{n \ge 0} H^0(X,\mathcal{O}_X(n))$. Then by the short exact sequence $0 \to \mathcal{I}_X \to \mathcal{O}_{\mathbb{P}^N_A} \to \mathcal{O}_X \to 0$, we have an injection $S \hookrightarrow R'$. By the definition of $R'$, it is contained in the field of fractions $L$ of $S$, and $R'$ is a finitely generated graded $S$-module. So $R'$ is Noetherian. Consider $\frac{s}{f} \in L$, where $s \in H^0(X,\mathcal{O}_X(l))$ and $f \in H^0(X,\mathcal{O}_X(m))$ with $f \ne 0$. Assume that $\frac{s}{f}$ is integral over $R'$. Then it can be shown that $l \ge m$ by looking at the monic polynomial defining $\frac{s}{f}$. Since $\mathcal{O}_{X,x}$ is a normal domain for $x \in X$, it follows that
$$
\Big(\frac{s}{f}\Big)_x \in \mathcal{O}_X(l-m)_x,
$$
which implies that $\frac{s}{f} \in H^0(X,\mathcal{O}_X(l-m))$. Hence $R'$ is a normal domain.

Next, we prove that $A=R'_0$, where $R'_0=H^0(X,\mathcal{O}_X)$. Denote by $X_K$ the generic fiber of $X \to \Spec A$ with $K:=\Frac(A)$. Then $X_K$ is a normal connected projective variety by assumption. First, assume the condition $(1)$.
Since $K$ is a field of characteristic zero, $X_{\overline{K}}:=X_K \times \Spec \overline{K}$ is a normal projective scheme defined over the algebraic closure $\overline{K}$ of $K$. Since $X_K$ is geometrically connected, $X_{\overline{K}}$ is a normal integral projective variety. Second, assume the condition $(2)$. Then $X_{\overline{K}}$ is an integral projective variety. In any case, we have
\begin{equation}
\label{connected1}
H^0(X_{\overline{K}},\mathcal{O}_{X_{\overline{K}}})=\overline{K}.
\end{equation}
Now $K':=H^0(X_K,\mathcal{O}_{X_K})$ is a finite field extension of $K$. By flat base change, we have
\begin{equation}
\label{connected2}
H^0(X_{\overline{K}},\mathcal{O}_{X_{\overline{K}}}) \simeq K' \otimes_K \overline{K}.
\end{equation}
Combining $(\ref{connected1})$ and $(\ref{connected2})$ together, we obtain $
\overline{K} \simeq K' \otimes_K \overline{K}$.
As this isomorphism holds only when $K'=K$, we have
$$
H^0(X_K,\mathcal{O}_{X_K})=K.
$$
Since there are inclusions $A \hookrightarrow H^0(X,\mathcal{O}_X) \hookrightarrow H^0(X_K,\mathcal{O}_{X_K})=K$, $A \hookrightarrow H^0(X,\mathcal{O}_X)$ is module-finite and $A$ is integrally closed in $K$, we have
$$
A=R'_0=H^0(X,\mathcal{O}_X).
$$
Recall that the direct summand of a normal domain is normal. So after replacing $R'$ by its appropriate Veronese subring $R$, it follows from \cite[Proposition 3 at page 159]{Bour} that there is a Noetherian graded normal domain $R$ satisfying $(\mathbf{St})$, $R_0=A$ and $X \simeq \Proj R$.
\end{proof}

Let us give a proof of the main theorem. Let $A$ be a discrete valuation ring and let $H$ be an element of the dual projective space $\mathbb{P}^n(A)^{\vee}$. Then we may identify $H$ as a hyperplane of $\mathbb{P}^n_A$ as a closed subscheme.

\begin{theorem}
\label{mainBertini}
Let $X$ be a normal connected projective scheme over $\Spec A$ such that $X \to \Spec A$ is flat and surjective, where $(A,\pi_A,k)$ is an unramified discrete valuation ring of mixed characteristic $p>0$. Assume that $\dim X \ge 2$, the generic fiber of $X \to \Spec A$ is geometrically connected and $k$ is an infinite field. 

Then there exists an embedding $X \hookrightarrow \mathbb{P}^n_A$, together with a Zariski-dense open subset $U \subset \mathbb{P}^n(k)^{\vee}$, where $\mathbb{P}^n(k)^{\vee}$ is the dual projective space of $\mathbb{P}^n(k)$, such that the scheme theoretic intersection $X \cap H$ is normal and flat over $A$ for any $H \in \Sp_A^{-1}(U)$, where $\Sp_A:\mathbb{P}^n(A)^{\vee} \to \mathbb{P}^n(k)^{\vee}$ is the specialization map (see Definition \ref{spmap}).
\end{theorem}

\begin{proof}
By Lemma \ref{normalgraded}, there is a Noetherian graded normal domain $R$ satisfying $(\mathbf{St})$, $R_0=A$ and $X \simeq \Proj R$. Take the homogeneous maximal ideal $\fm=\pi_A A \oplus \bigoplus_{n>0} R_n$ of $R$, together with the localized pointed affine cone $C^0_+(R)=\Spec R_{\fm} \setminus j(\Spec A)=\Spec R_{\fm} \setminus \{\fm_+ R_{\fm},\fm R_{\fm}\}$, where $j:\Spec A \hookrightarrow \Spec R_{\fm}$ is the closed immersion induced by $R_{\fm} \twoheadrightarrow R_{\fm}/\fm_+R_{\fm}=A$. Then $R_{\fm}$ is a Noetherian local ring such that $\pi_A \notin \fm^2R_{\fm}$. There is a flat surjective smooth morphism (see Proposition \ref{Cone}):
$$
h:C^0_+(R)=\Spec R_{\fm} \setminus j(\Spec A) \to X.
$$

Let us fix a system of minimal generators $x_1,\ldots,x_d \in R_1$ such that $\fm=(\pi_A,x_1,\ldots,x_d)$ and hence $\fm_+=(x_1,\ldots,x_d)$. Since $A \simeq R/(x_1,\ldots,x_d)$, $\fm_+=(x_1,\ldots,x_d)$ is a prime ideal such that $\Ht(x_1,\ldots,x_d)=\dim R-1$. Let us pick a point $\alpha=(\alpha_1:\cdots:\alpha_d) \in \mathbb{P}^{d-1}(A)$ and let $H_{\alpha}$ be the hyperplane of $\mathbb{P}^{d-1}_A$ defined by the homogeneous polynomial: $\widetilde{\alpha}_1X_1+\cdots +\widetilde{\alpha}_d X_d \in A[X_1,\ldots,X_d]$, to which we give the standard grading: $\deg(X_1)=\cdots=\deg(X_d)=1$. There is a surjection of graded rings: $A[X_1,\ldots,X_d] \twoheadrightarrow R$ by mapping $X_i$ to $x_i$. This defines a closed immersion $X \hookrightarrow \mathbb{P}^{d-1}_A$. Letting $n:=d-1$, we shall prove that $X \hookrightarrow \mathbb{P}_A^{d-1}$ has all the properties as stated in the theorem. 

Let us put
$$
\mathbf{x}_{\widetilde{\alpha}}:=\sum_{i=1}^d\widetilde{\alpha}_ix_i
$$
and consider the following fiber product diagram:
$$
\begin{CD}
C^0_+(R/\mathbf{x}_{\widetilde{\alpha}}R) @>>> X \cap H_{\alpha} \\
@VVV @VVV\\
C^0_+(R) @>h>> X \\
\end{CD}
$$
where all the vertical maps are natural closed immersions. Then by Proposition \ref{Cone}, it follows that
\begin{equation}
\label{cri1}
X \cap H_{\alpha}~\mbox{is normal}.~\iff~C^0_+(R/\mathbf{x}_{\widetilde{\alpha}}R)~\mbox{is normal.}
\end{equation}
Moreover, we have
\begin{equation}
\label{cri2}
C^0_+(R/\mathbf{x}_{\widetilde{\alpha}}R)~\mbox{is normal}.~\iff~R_{\fm}/\mathbf{x}_{\widetilde{\alpha}}R_{\fm}~\mbox{is normal after localization at}~\fp \in \Delta,
\end{equation}
where we put
$$
\Delta:=V(\mathbf{x}_{\widetilde{\alpha}}R_{\fm}) \cap C^0_+(R) \subset \Spec R_{\fm}.
$$
We find that $\Delta=V(\mathbf{x}_{\widetilde{\alpha}} R_{\fm}) \cap C^0_+(R)=V(\mathbf{x}_{\widetilde{\alpha}} R_{\fm}) \setminus \{\fm_+ R_{\fm},\fm R_{\fm}\}$. Now consider the natural map $R_{\fm} \to R_{\fm}^{\wedge}$, where $R_{\fm}^{\wedge}$ is the $\fm R_{\fm}$-adic completion of $R_{\fm}$. Then this is a flat local map between Noetherian local rings. Let $A^{\wedge}$ be the $(\pi_A)$-adic completion of $A$. Since $R_{\fm}$ is essentially of finite type over the discrete valuation ring $A$ whose field of fractions is of characteristic zero, it is an excellent local normal domain. Hence $R_{\fm}^{\wedge}$ is a complete local normal domain with $A^{\wedge}$ its coefficient ring (by our assumption, $A$ is an unramified discrete valuation ring). It defines an affine scheme map:
$$
g:\Spec R_{\fm}^{\wedge} \to \Spec R_{\fm}.
$$
Moreover, we have a commutative diagram of specialization maps:
$$
\begin{CD}
\mathbb{P}^{d-1}(A) @>>> \mathbb{P}^{d-1}(A^{\wedge}) \\
@V\Sp_A VV @V\Sp_{A^{\wedge}} VV \\
\mathbb{P}^{d-1}(k) @>=>> \mathbb{P}^{d-1}(k) \\
\end{CD}
$$

Let us pick a prime $\fq \in \Spec R_{\fm}^{\wedge}$ and put $\fp:=g(\fq) \in \Spec R_{\fm}$. Then since $g$ is a regular local map, we have
$$
R_{\fm}/\mathbf{x}_{\widetilde{\alpha}}R_{\fm}~\mbox{is normal after localization at}~\fp. \iff R_{\fm}^{\wedge}/\mathbf{x}_{\widetilde{\alpha}}R_{\fm}^{\wedge}~\mbox{is normal after localization at}~\fq.
$$
Hence in view of $(\ref{cri1})$ and $(\ref{cri2})$, it suffices to prove that there exists a Zariski-dense open subset $U \subset \mathbb{P}^{d-1}(k)$ such that the following holds:
\begin{equation}
\label{cri3}
R_{\fm}^{\wedge}/\mathbf{x}_{\widetilde{\alpha}}R_{\fm}^{\wedge}
~\mbox{is normal after localization at}~\fq \in g^{-1}(\Delta)~\mbox{for}~\alpha \in \Sp^{-1}_{A^{\wedge}}(U).
\end{equation}

So let us prove $(\ref{cri3})$ below. By elementary set theory, we have
$$
g^{-1}(\Delta)=V(\mathbf{x}_{\widetilde{\alpha}}R_{\fm}^{\wedge}) \cap g^{-1}\big(C^0_+(R)\big) \subset \Spec R_{\fm}^{\wedge}
$$
and $g^{-1}(\Delta)$ is a locally closed subset of $\Spec R_{\fm}^{\wedge}$. The set of minimal primes in ${\rm{Sing}}(R_{\fm}^{\wedge})$ is finite, where ${\rm{Sing}}(R_{\fm}^{\wedge})$ denotes the singular locus of
$\Spec R_{\fm}^{\wedge}$. By recalling that $\fm=(\pi_A,x_1,\ldots.x_d)$ and $\fm_+=(x_1,\ldots,x_d)$, it follows that $\{\fm_+ R^{\wedge}_{\fm}, \fm R^{\wedge}_{\fm}\}$ is the set of all prime ideals containing $\fm_+ R^{\wedge}_{\fm}$. Then we have $\{\fm_+ R^{\wedge}_{\fm}, \fm R^{\wedge}_{\fm}\} \cap g^{-1}(\Delta)=\emptyset$, because we know $\{\fm_+ R^{\wedge}_{\fm}, \fm R^{\wedge}_{\fm}\} \cap g^{-1}\big(C^0_+(R)\big)=\emptyset$.

To prove $(\ref{cri3})$, we need to check Serre's $(R_1)$ and $(S_2)$. Taking the open subset $U' \subset \mathbb{P}^{d-1}(k)$ as in Proposition \ref{symbolic}, it follows from Remark \ref{symbolic2} that
\begin{equation}
\label{Reg}
\Reg(R_{\fm}^{\wedge}) \cap V(\mathbf{x}_{\widetilde{\alpha}}R_{\fm}^{\wedge}) \subset \Reg(R_{\fm}^{\wedge}/
\mathbf{x}_{\widetilde{\alpha}}R_{\fm}^{\wedge})~\mbox{holds at every point}~\fq \in g^{-1}(\Delta)
\end{equation}
$$
\mbox{for every}~\alpha \in \Sp^{-1}_{A^{\wedge}}(U') \subset \mathbb{P}^{d-1}(A^{\wedge}).
$$
In other words, if we have $\fq \in \Reg(R_{\fm}^{\wedge}) \cap V(\mathbf{x}_{\widetilde{\alpha}}R_{\fm}^{\wedge}) \cap g^{-1}(\Delta)$, then $\fq \in \Reg(R_{\fm}^{\wedge}/\mathbf{x}_{\widetilde{\alpha}}R_{\fm}^{\wedge})$.

Let us put
$$
Q_1=\{\fp \in g^{-1}(\Delta)~|~\fp~\mbox{is a minimal prime in}~{\rm{Sing}}(R_{\fm}^{\wedge})\}.
$$
This is a finite set. Notice that every prime contained in $Q_1$ has height at least 2 due to the $(R_1)$ condition on $R_{\fm}^{\wedge}$. Let us put
$$
Q_2=\{\fp \in g^{-1}(\Delta)~|~\depth (R^{\wedge}_{\fm})_{\fp}=2~\mbox{and}~\dim (R^{\wedge}_{\fm})_{\fp}>2\}.
$$ 
Since $R_{\fm}^{\wedge}$ is a complete local domain, we find that the set of primes $\fp$ such that $(R_{\fm}^{\wedge})_{\fp}$ satisfies Serre's $(S_3)$ forms an open subset of $\Spec R_{\fm}^{\wedge}$ by \cite[Proposition (6.11.2)]{Gr65}. Then the $(S_2)$ condition on $R_{\fm}^{\wedge}$ allows us to apply \cite[Lemma 3.2]{Fl77} to conclude that $Q_2$ is a finite set. Set $Q_1 \cup Q_2=\{\fp_1,\ldots,\fp_m\}$. By the same reasoning as in the Step 1 of the proof of \cite[Theorem 4.4]{OS15} and proceeding as in the Step 2 of the proof of \cite[Theorem 4.4]{OS15} together with $(\ref{Reg})$, we can attach to
each $\fp_i$ a Zariski-dense open subset $U_i \subset \mathbb{P}^{d-1}(k)$ with the properties stated below: Let us put
$$
U'':=\bigcap_{i=1}^m U_i \subset \mathbb{P}^{d-1}(k).
$$
Then we obtain the following assertion:
$$
R^{\wedge}_{\fm}/\mathbf{x}_{\widetilde{\alpha}}R^{\wedge}_{\fm}~
\mbox{is normal after localization at}~\fq \in g^{-1}(\Delta)~\mbox{for}~\alpha \in \Sp^{-1}_{A^{\wedge}}(U' \cap U'').
$$
This shows that the desired conclusion $(\ref{cri3})$ is fulfilled
by letting $U:=U' \cap U''$. This finishes the proof of the theorem.
\end{proof}

We have the following corollary.

\begin{corollary}
\label{normalCartier}
Let $X$ be a normal connected affine scheme such that there is a surjective flat morphism of finite type $X \to \Spec A$, where $A$ is an unramified discrete valuation ring of mixed characteristic $p>0$. Assume that $\dim X \ge 2$, the generic fiber of $X \to \Spec A$ is geometrically connected and the residue field of $A$ is infinite. 

Then there exists an infinite family of pairwise distinct effective Cartier divisors $\{D_{\lambda}\}_{\lambda \in \Lambda}$ of $X$ such that each $D_{\lambda}$ is normal and flat over $A$.
\end{corollary}

\begin{proof}
We can embed $X$ into $\mathbb{P}^n_A$ as a dense open subset of an integral projective scheme $Y$. Let us consider the normalization $h:\overline{Y} \to Y$. Then $\overline{Y}$ is a normal connected projective scheme over $\Spec A$ and $h^{-1}(X) \simeq X$ because $X \subset \Nor(Y)$. Hence, we may assume that $X$ is embedded into a normal connected projective scheme $Y$ as a dense open subset. Since the generic fiber of $X \to \Spec A$ is geometrically connected, so is the generic fiber of $Y \to \Spec A$.

Since the residue field of $k$ is infinite, every non-empty Zariski open subset of $\mathbb{P}^n(k)$ is infinite. Now we can apply Theorem \ref{mainBertini} to the normal connected projective scheme $Y$ and find a family of pairwise distinct effective Cartier divisors $\{D'_{\lambda}\}_{\lambda \in \Lambda}$ of $Y$ such that $D'_{\lambda}$ is normal and flat over $A$ and $X \cap D'_{\lambda} \ne \emptyset$, because $X \hookrightarrow Y$ is an open immersion. Then, the set $\{D_{\lambda}:=X \cap D'_{\lambda}\}_{\lambda \in \Lambda}$ is the one as desired.
\end{proof}

\begin{remark}
\label{normalCartier2}
\begin{enumerate}
\item
We have stated and proved Theorem \ref{mainBertini} only when $A$ is an unramified discrete valuation ring. Suppose that we are given a flat and finite type morphism $X \to \Spec B$ such that $B$ is a ramified discrete valuation ring. If there is a finite morphism $\Spec B \to \Spec A$ such that $A$ is an unramified discrete valuation ring and the generic fiber of $X \to \Spec A$ is geometrically connected, Corollary \ref{normalCartier} still applies to the composite morphism $X \to \Spec B \to \Spec A$. However, since $\Spec B \to \Spec A$ is ramified, it does not happen that the generic fiber of $X \to \Spec A$ is geometrically connected. Also the authors do not know if any ramified discrete valuation ring admits such an unramified discrete valuation ring.

\item
The assumption on the geometric connectedness of the generic fiber of $X \to \Spec A$ seems to be unavoidable in the proof of Theorem \ref{mainBertini}. Fix a prime $p \ne 2$ and consider the following example:
$$
\Proj\big(\mathbb{Z}_p[x,y]/(x^2-py^2)\big) \to \Spec \mathbb{Z}_p
$$
with $\deg(x)=\deg(y)=1$. Then $\mathbb{Z}_p[x,y]/(x^2-py^2)$ is not integrally closed and we get the normalization map
$$
\mathbb{Z}_p[x,y]/(x^2-py^2) \hookrightarrow \mathbb{Z}_p[x,y,\frac{x}{y}]/(x^2-py^2) \simeq \mathbb{Z}_p[\sqrt{p}][t]=:R
$$
by letting $\frac{x}{y} \mapsto \sqrt{p}$, $x \mapsto \sqrt{p}t$ and $y \mapsto t$. The graded normal domain $R$ is as shown in Lemma \ref{normalgraded}. However, we have $R_0=\mathbb{Z}_p[\sqrt{p}]$.
\end{enumerate}
\end{remark}

We also obtain the Bertini theorem for regular projective schemes.

\begin{corollary}
\label{corBertini}
Let $X$ be a regular connected projective scheme over $\Spec A$ such that $X \to \Spec A$ is flat and surjective, where $(A,\pi_A,k)$ is an unramified discrete valuation ring of mixed characteristic $p>0$. Assume that $\dim X \ge 2$, the generic fiber of $X \to \Spec A$ is geometrically connected and $k$ is an infinite field. 

Then there exists an embedding $X \hookrightarrow \mathbb{P}^n_A$, together with a Zariski-dense open subset $U \subset \mathbb{P}^n(k)^{\vee}$, where $\mathbb{P}^n(k)^{\vee}$ is the dual projective space of $\mathbb{P}^n(k)$, such that the scheme theoretic intersection $X \cap H$ is regular and flat over $A$ for any $H \in \Sp_A^{-1}(U)$.
\end{corollary}

\begin{proof}
The proof is modeled after that of Theorem \ref{mainBertini} and let us only sketch the idea. As $X$ is a regular scheme, it is normal. Keeping the notation as in the proof of Theorem \ref{mainBertini}, we can find a Noetherian graded normal domain $R$ such that $R$ satisfies $(\mathbf{St})$ and $X \simeq \Proj R$. We have a flat surjective morphism $C_+^0(R)=\Spec R_{\fm} \setminus \{\fm_+ R_{\fm},\fm R_{\fm}\} \to X=\Proj R$ with smooth fibers by Proposition \ref{Cone}. Then
$$
X \cap H_{\alpha}~\mbox{is regular}. \iff R_{\fm}/\mathbf{x}_{\widetilde{\alpha}}R_{\fm}~\mbox{is regular after localization at}~\fp \in \Delta.
$$
So after completing $R_{\fm}$, an application of Proposition \ref{symbolic} together with Remark \ref{symbolic2} gives infinitely many regular hyperplane sections of $X$.
\end{proof}

\section{Applications to Weil divisor class groups of normal schemes}

Let us recall the following theorem.

\begin{theorem}
\label{classicalBertini}
Let $X \subset \mathbb{P}^n_k$ be an irreducible and reduced projective variety over an algebraically closed field $k$ such that $\dim X \ge 2$. Then there exists a Zariski-dense open subset $U \subset \mathbb{P}^n(k)^{\vee}$ such that $X \cap H$ is irreducible and reduced for every $H \in U$.
\end{theorem}

\begin{proof}
Since the ground field $k$ is assumed to be algebraically closed, if $K(X)$ denotes the function field of $X$, the extension of fields $K(X)/k$ satisfies the assumption of \cite[Theorem 11]{S50}. Hence $X \cap H$ is reduced and irreducible for all hyperplanes $H$ belonging to a Zariski-dense open subset of $\mathbb{P}^n(k)^{\vee}$. The reducedness property of $X \cap H$ also follows from \cite[Corollary 2]{CGM}.
\end{proof}

\begin{remark}
If one takes a ``hypersurface'' instead of a ``hyperplane'' in Theorem \ref{classicalBertini} in the positive characteristic case, one cannot get even reduced Cartier divisors as shown in the following example. Let $k$ be a perfect field of characteristic $p>0$ and let $\mathbb{P}^n_k=\Proj\big(k[X_0,\ldots,X_n]\big)$. For $\underline{\alpha}=(\alpha_0:\cdots:\alpha_n) \in \mathbb{P}^n(k)$. Let $H_{\underline{\alpha}} \subset \mathbb{P}^n_k$ be the hypersurface defined by the homogeneous polynomial $\alpha_0X_0^p+\cdots+\alpha_n X_n^p$. Then the Cartier divisor $H_{\underline{\alpha}}$ cannot be reduced for any choice of $\underline{\alpha}$.
\end{remark}

Let $X$ be a Noetherian integral scheme with a Cartier divisor $D$. Since $D \hookrightarrow X$ is a regular immersion of codimension one, we can define the \textit{Gysin map} $\CH_i(X) \to \CH_{i-1}(D)$, where $\CH_i(X)$ is the \textit{Chow group} of algebraic cycles of dimension $i$ on $X$ (see \cite[(2.6) or (6.2)]{Ful}). When $X$ is a normal scheme and $\dim X=d$, then we have $\Cl(X)=\CH_{d-1}(X)$, where $\Cl(X)$ denotes the \textit{Weil divisor class group} of $X$. So we get the restriction map of Weil divisor class groups $\Cl(X) \to \Cl(D)$. Let us prove the mixed characteristic version of Grothendieck-Lefschetz theorem by Ravindra and Srinivas which is stated below \cite[Theorem 1]{RaSr06}.

\begin{theorem}[Ravindra-Srinivas]
\label{RaSr}
Let $X$ be a normal connected projective variety defined over an algebraically closed field of characteristic zero. Let $|V|$ be a base-point free linear system associated to a linear subspace $V \subset H^0(V,\mathcal{L})$ for an ample invertible sheaf $\mathcal{L}$. Then there exists a Zariski-dense open subset $U \subset |V|$ such that for $Y \in U$, $Y$ is a normal connected variety and the restriction map of the Weil divisor class groups
$$
\Cl(X) \to \Cl(Y)
$$
is an isomorphism if $\dim X \ge 4$, and injective with finitely generated cokernel if $\dim X=3$. 
\end{theorem}

The book \cite{H70} is a good introduction to the study of divisor class groups of projective varieties. \cite{BN14} is an excellent survey which discusses relevant topics from the viewpoint of Hodge theory begun by Griffiths and his collaborators.

\begin{corollary}
\label{RavindraSrinivas}
Let $X \to \Spec A$ be a surjective projective flat morphism, where $(A,\pi_A,k)$ is an unramified discrete valuation ring with mixed characteristic and algebraically closed residue field. Assume that
$X$ is a normal scheme, the generic fiber is geometrically connected and if $X_k=\sum_{i=1}^t m_i \Gamma_i$ denotes the closed fiber of $X \to \Spec A$, then each $\Gamma_i$ is a principal, irreducible and reduced divisor. 

Then there exists an infinite family of pairwise distinct effective Cartier divisors $\{D_{\lambda}\}_{\lambda \in \Lambda}$ of $X$ such that each $D_{\lambda}$ is normal, connected, and the restriction map of the Weil divisor class groups
$$
\Cl(X) \to \Cl(D_{\lambda})
$$
is an isomorphism if $\dim X \ge 5$, and injective with finitely generated cokernel if $\dim X=4$. 
\end{corollary}

We should note that the closed fiber $X_k$ is not necessarily a normal scheme.

\begin{proof}
Let $K$ be the field of fractions of $A$. It is a field of characteristic zero. Let $X_K$ be the generic fiber of $X \to \Spec A$. Then $X_K$ is a geometrically connected normal projective variety over $K$ and $\dim X_K=\dim X-1$. Let 
us take $X \hookrightarrow \mathbb{P}_A^n$ as in Theorem \ref{mainBertini}. Then there is a commutative diagram:
$$
\begin{CD}
X_k @>\hookrightarrow>> \mathbb{P}^n_k \\
@VVV @VVV \\
X @>\hookrightarrow>> \mathbb{P}_A^n \\
@AAA @AAA \\
X_K @>\hookrightarrow>> \mathbb{P}_K^n
\end{CD}
$$
Applying \cite[Proposition 11.40]{GW10} to the pair $(X,X_K)$, there is an exact sequence of groups
$$
\sum_{i=1}^t\mathbb{Z} [\Gamma_i] \to \Cl(X) \to \Cl(X_K) \to 0.
$$
Since $\Gamma_i$ is a principal divisor for $i=1,\ldots,t$ by assumption, the image of $[\Gamma_i]$ in $\Cl(X)$ is trivial. Hence we have an isomorphism:
\begin{equation}
\label{isom1}
\Cl(X) \simeq \Cl(X_K).
\end{equation}
By Theorem \ref{mainBertini}, there exists a Zariski-dense open subset $U \subset \mathbb{P}^n(k)^{\vee}$ such that $X \cap H$ is a normal Cartier divisor for every $H \in \mathcal{U}:=\Sp_A^{-1}(U) \subset \mathbb{P}^n(A)^{\vee}$.

On the other hand, applying Theorem \ref{RaSr} together with remarks at page 584 of \cite{RaSr06}, we have the following assertion:
\begin{enumerate}
\item[-]
There exists a Zariski-dense open subset $\mathcal{V} \subset \mathbb{P}^n(K)^{\vee}=\mathbb{P}^n(A)^{\vee}$ such that $X_K \cap H$ on $X_K$ is connected and normal for every $H \in \mathcal{V}$ and the restriction map $\Cl(X_K) \to \Cl(X_K \cap H)$ is an isomorphism if $\dim X_K \ge 4$, and injective with finitely generated cokernel if $\dim X_K=3$. 
\end{enumerate}
It is clear that $\Lambda:=\mathcal{U} \cap \mathcal{V} \subset \mathbb{P}^n(A)^{\vee}$ is an infinite set. 

Denote by $H_{\lambda}$ the hyperplane of $\mathbb{P}^n_A$ corresponding to $\lambda \in \Lambda$. Then we have an infinite family of pairwise distinct effective Cartier divisors $\{D_{\lambda}:=X \cap H_{\lambda}\}_{\lambda \in \Lambda}$ on $X$. Since $X_K \cap D_{\lambda}=X_K \cap H_{\lambda}$ is connected and $D_{\lambda}$ is flat over $\Spec A$, the divisor $D_{\lambda}$ is the Zariski closure of $X_K \cap D_{\lambda}$ in $X$ and thus, $D_{\lambda}$ is connected. It follows by the normality of $D_{\lambda}$ that each $D_{\lambda}$ is irreducible and reduced. After shrinking $\Lambda \subset \mathbb{P}^n(A)^{\vee}$, we can assume that $\Cl(D_{\lambda}) \to \Cl(X_K \cap D_{\lambda})$ is an isomorphism. Indeed, it suffices to choose $D_{\lambda}$ so that the following assertion holds by applying Theorem \ref{classicalBertini} to each $\Gamma_i$ (notice that $\Gamma_i \cap D_{\lambda}$ is a
hyperplane section of $\Gamma_i \subset \mathbb{P}^n_k$).
\begin{enumerate}
\item[-]
The closed fiber of the map $D_{\lambda} \to \Spec A$ is given as $X_k \cap D_{\lambda}=\sum_{i=1}^t m_i (\Gamma_i \cap D_{\lambda})$ and each $\Gamma_i \cap D_{\lambda}$ is irreducible and reduced.
\end{enumerate}
Then, since the restriction of a principal Cartier divisor is principal, it follows that $\Gamma_i \cap D_{\lambda}$ is a principal, irreducible and reduced divisor of $D_{\lambda}$. Since the generic fiber of $D_{\lambda} \to \Spec A$ is $X_K \cap D_{\lambda}$, applying \cite[Proposition 11.40]{GW10}, we have 
$$
\sum_{i=1}^t\mathbb{Z} [\Gamma_i \cap D_{\lambda}] \to \Cl(D_{\lambda}) \to \Cl(X_K \cap D_{\lambda}) \to 0
$$
and thus
\begin{equation}
\label{isom2}
\Cl(D_{\lambda}) \simeq \Cl(X_K \cap D_{\lambda}).
\end{equation}

Combining $(\ref{isom1})$ and $(\ref{isom2})$ together, we have the commutative diagram of the Weil divisor class groups:
$$
\begin{CD}
\Cl(X_K) @>>> \Cl(X_K \cap D_{\lambda}) \\
@| @| \\
\Cl(X) @>>> \Cl(D_{\lambda}) \\
\end{CD}
$$
Since the upper horizontal map has all the required properties as in Theorem \ref{RaSr}, the lower horizontal map fulfills the similar properties.
\end{proof}

\begin{acknowledgement}
The authors are grateful to the referee for pointing out errors and providing remarks.
\end{acknowledgement}

\end{document}